\documentclass[11pt]{amsart}
\usepackage[margin=1in]{geometry}
\usepackage{amssymb}
\usepackage{microtype}
\usepackage[colorlinks=true,linkcolor=blue,citecolor=blue,urlcolor=blue]{hyperref}

\newtheorem{theorem}{Theorem}[section]
\newtheorem{lemma}[theorem]{Lemma}
\newtheorem{proposition}[theorem]{Proposition}

\newtheorem{question}[theorem]{Question}
\theoremstyle{definition}

\theoremstyle{remark}
\newtheorem{remark}[theorem]{Remark}

\newcommand{\Z}{\mathbb Z}
\newcommand{\R}{\mathbb R}
\newcommand{\N}{\mathbb N}
\newcommand{\calD}{\mathcal D}
\newcommand{\calR}{\mathcal R}
\newcommand{\eps}{\varepsilon}
\newcommand{\tY}{\widetilde Y}
\newcommand{\hY}{\widehat Y}
\DeclareMathOperator{\Free}{Free}
\DeclareMathOperator{\supp}{supp}
\DeclareMathOperator{\Lip}{Lip}
\DeclareMathOperator{\dist}{dist}

\DeclareMathOperator{\relint}{relint}

\title[Continuous homomorphisms of higher dimension]
{Continuous graph homomorphisms of higher dimensional abelian group actions}
\author{Ruijun Wang}
\date{}

\begin{document}
\begin{abstract}
For every fixed integer $d\geq2$, we prove that the finite graphs receiving
a continuous homomorphism from the standard Schreier graph $F(2^{\mathbb Z^d})$ form a $\Sigma^0_1$-complete set. This extends a theorem of Gao, Jackson, Krohne and Seward when
$d=2$. For the positive part of the reduction, we show that if a graph $H$ satisfies a certain property then there is a continuous graph homomorphism from $F(2^{\mathbb Z^d})$ to $H$, in particular, the complete graph $K_4$ satisfies this property. This extends a theorem of Gao and Jackson. We also prove that $\big (\mathcal H_d\big )$ is a strictly decreasing sequence.
\end{abstract}
\maketitle

\section{Introduction}
In the seminal paper \cite{KST}, Kechris, Solecki and Todorcevic initiated the study of descriptive combinatorics. In \cite{GJ}, Gao and Jackson developed the rectangular partition for Schreier graphs of countable abelian groups actions, in \cite[Theorem 4.2]{GJ}, the authors showed the standard Schreier graphs $F(2^{\mathbb Z^d}),\ d>1$ have a continuous proper $4$-coloring. In \cite{GJKS}, Gao, Jackson, Krohne and Seward studied the continuous combinatorics of abelian group actions, in \cite[Theorem 1.5.1]{GJKS} and \cite[Theorem 3.1.1]{GJKS}, the authors showed two different proofs that $F(2^{\mathbb Z^2})$ does not admit a continuous proper $3$-coloring. As a corollary, the continuous chromatic number $\chi_c(F(2^{\mathbb Z^d}))=4,\ d>1$.

A proper coloring is a graph homomorphism to the complete graphs, so it gives rise to the graph homomorphism problems. In \cite[Section 4]{GJKS}, the authors asked the following question.

\begin{question}[The continuous (or Borel) graph homomorphism problem, Gao, Jackson, Krohne and Seward]
For which finite graphs $H$ does there exist a continuous (or Borel) graph homomorphism from $F(2^{\mathbb Z^d})$ to $H$? What is the complexity of this set? In particular, is it computable?
\end{question}

In \cite{CU}, Chandgotia and Unger proved that there is a Borel graph homomorphism from the Schreier graph $F(2^{\mathbb Z^d}),\ d>1$ to any finite graph with an odd-cycle. It is well-known that the graph does not admit a Borel proper $2$-coloring. So a finite graph receives a Borel graph homomorphism from $F(2^{\mathbb Z^d}),\ d>1$ if and only if it is not bipartite. In particular, the set is computable.

In \cite[Section 4.2]{GJKS}, the authors proved that the continuous subshift problem for $F(2^{\mathbb Z})$ is computable. As a consequence, the one-dimensional continuous graph homomorphism problem is computable. In \cite[Section 4.3]{GJKS}, the authors proved that the continuous subshift problem for $F(2^{\mathbb Z^2})$ is $\Sigma^0_1$-complete, and claimed without a proof that the continuous subshift problem for $F(2^{\mathbb Z^d}),\ d>1$ is $\Sigma^0_1$-complete. In \cite[Theorem 4.4.1]{GJKS}, the authors proved that the set of finite graphs receiving a continuous graph homomorphism from $F(2^{\mathbb Z^2})$ is $\Sigma^0_1$-complete, in particular, it is not computable.

In this paper, we show the following theorem.

\begin{theorem}\label{intro:main}
For every fixed integer $d\geq2$, the set of finite graphs receiving a continuous graph homomorphism from $F(2^{\mathbb Z^d})$ is
$\Sigma^0_1$-complete.
\end{theorem}

As a corollary, the continuous subshift problem for $F(2^{\mathbb Z^d}),\ d>1$ is $\Sigma^0_1$-complete. The reduction uses the same reduction in the proof of the following theorem, see \cite[Section 3.3, 3.4, 3.5 and 4.4]{GJKS}.

\begin{theorem}[Gao, Jackson, Krohne and Seward]\label{intro:GJKS}
There is a $\Sigma^0_1$-complete set $A\subseteq\N$ and a computable
sequence of finite connected simple graphs $(H_n)_{n\in\N}$ such that:
\begin{enumerate}
\item if $n\in A$, then $H_n$ has an odd closed walk $\zeta_n$ with
$[\zeta_n]^2=1$ in $\Pi_1^*(H_n)$;
\item if $n\notin A$, then $F(2^{\mathbb Z^2})$ has no continuous homomorphism
to $H_n$.
\end{enumerate}
\end{theorem}

We write
\[
 \mathcal H_d=\{H:F(2^{\mathbb Z^d})\longrightarrow_c H\}.
\]

It is easy to see that $\mathcal H_j\subseteq \mathcal H_i$ for $i<j$, see the proof of \cite[Corollary 1.5.2]{GJKS} for example. So we only need to strengthen the positive part of the reduction, that is the following theorem.

\begin{theorem}\label{intro:positive}
Let $H$ be a finite simple graph containing an odd closed walk $\gamma$
such that $[\gamma]^2=1$ in $\Pi_1^*(H)$.  For every $d\geq1$, every
free continuous $\Z^d$-action on a zero-dimensional Polish space admits
a continuous homomorphism from its standard Schreier graph to $H$.
\end{theorem}

$\Pi_1^*(H)$ is called the \emph{reduced homotopy group}. It is defined in \cite[Section 3.3]{GJKS}. It is a quotient group of the fundamental group $\Pi_1(H)$ taking a closed walk of length 4 as the identity element. The length of a closed walk defines a group homomorphism from $\Pi_1^*(H)$ to $\mathbb Z/2\mathbb Z$. For example, the fundamental group $\Pi_1(K_3)$ is isomorphic to $\mathbb Z$, and every closed walk of length four in $K_3$ is already homotopically trivial, so $\Pi_1^*(K_3)$ is also isomorphic to $\mathbb Z$, hence there is no odd closed walk $\gamma$ with $[\gamma]^2=1$.

For another example, in the reduced homotopy group $\Pi_1^*(K_4)$, we define three closed walks of length 3.

$$\alpha=(x,y,z,x),\ \beta=(x,y,w,x),\ \gamma=(x,z,w,x).$$We compute

$$\alpha\gamma=(x,y,z,w,x)=1,\ \beta\gamma^{-1}=(x,y,w,z,x)=1,\ \alpha^{-1}\beta=(x,z,y,w,x)=1$$
where $1$ is the identity element.

So $\alpha=\beta=\gamma$ and $\alpha^2=1$. In fact, $\Pi_1^*(K_4)\cong \mathbb Z/2\mathbb Z$. The graph $K_4$ satisfies the above assumption, so $F(2^{\mathbb Z^d}),\ d>1$ admits a continuous proper $4$-coloring. This theorem extends \cite[Theorem 4.2]{GJ}.

Above all, this proves $\Sigma^0_1$-hardness of the main
Theorem~\ref{intro:main}. In \cite[Theorem 1.14]{Bernshteyn},
Bernshteyn proved that, for any countably infinite marked group
$(\Gamma,S)$, there is
an explicit countable family $\mathcal W$ of finite graphs such that
a finite graph $H$ receives a continuous homomorphism from the
Schreier graph of $\Free(2^\Gamma)$ if and only if some graph in
$\mathcal W$ admits a graph homomorphism to $H$. For
$\Gamma=\mathbb Z^d$ with its standard generators, the construction
in \cite[Section 5.B]{Bernshteyn} gives a computable enumeration of
$\mathcal W$. Since homomorphism existence between two finite graphs
is decidable, this shows that $\mathcal H_d$ is computably enumerable.
This finishes the proof of the main theorem.

We remark that the sequence $\mathcal H_d\supseteq \mathcal H_{d+1}$ is strictly decreasing. We remark that the reverse of Theorem \ref{intro:positive} is not true in general. We will show these remarks in the last section. The rest of this paper is devoted to a proof of Theorem \ref{intro:positive}.

This paper is a collaboration with artificial intelligence. The AI model is gpt 6. The AI proves the main theorem under human suggestions and writes the main body of early drafts. The author writes the abstract and induction and improves readability. The author has reviewed all AI-generated content.

\section{Preliminaries}

\subsection{Actions and graph homomorphisms}
Let $G=\Z^d$, with standard generators $e_1,\ldots,e_d$.
All actions on the input space $X$ are continuous actions on
zero-dimensional Polish spaces. A space is zero-dimensional if it
has a base of clopen sets, and a set is clopen if it is both open
and closed. An action is free if $g\cdot x=x$ implies $g=0$.
Each orbit of a free action can therefore be identified with
$\Z^d$ by choosing an origin $x$ and sending $g$ to $g\cdot x$.

We use the shift action
\[
 (g\cdot x)(u)=x(u+g),\qquad x\in2^G,\quad g,u\in G,
\]
and write $X_d=\Free(2^G)$ for its free part. The standard
Schreier graph on a free $G$-space $X$ joins $x$ to
$e_i\cdot x$ for $1\leq i\leq d$. A continuous homomorphism to a
finite graph $H$ is a continuous map $c:X\to V(H)$ taking adjacent
points to adjacent vertices. The target vertex set has the discrete
topology, so continuity means that every color class is clopen.
The notation $F(2^{\Z^d})\longrightarrow_c H$ refers to a continuous
homomorphism from this standard graph on $X_d$.

We will also use continuous universality: if $H\in\mathcal H_d$,
then every free continuous $\Z^d$-action on a zero-dimensional
Polish space admits a continuous homomorphism to $H$
\cite[Theorem 1.12]{Bernshteyn}. The Cartesian product
$J_1\square\cdots\square J_k$ of graphs has vertex set
$V(J_1)\times\cdots\times V(J_k)$; two tuples are adjacent
when they agree in all but one coordinate, where their entries
are adjacent in that factor.

Every finite binary pattern on $\Z^d$ has a free extension.
Enumerate the nonzero displacements $g$, and at each stage choose
two still unassigned sites $u,u+g$ and give them different symbols.
Fill the remaining sites arbitrarily. The resulting configuration
has no nonzero period. Thus $X_d$ is dense in $2^{\Z^d}$. We also
recall that a homeomorphism $T$ is topologically transitive if, for
any two nonempty open sets $U,V$, some iterate $T^nU$ meets $V$.

We work with finite simple graphs, coded by their finite adjacency
matrices. Allowing loops does not change the complexity conclusion:
a graph with a loop admits a constant homomorphism, and the simple
graphs already provide the hardness reduction. A set of graph codes
is $\Sigma^0_1$ if it is computably enumerable. It is
$\Sigma^0_1$-complete if every computably enumerable subset of
$\N$ computably many-one reduces to it.

\subsection{Walks and the reduced homotopy group}
Here we give the graph-theoretic description of the reduced homotopy
group used in \cite[Section 3.3, pp.\ 83--84]{GJKS}.
A walk is a finite sequence
$\sigma=(v_0,\ldots,v_m)$ with consecutive vertices adjacent.
Repeated vertices and edges are allowed, and the length is $m$.
The length-zero walk $(v)$ is constant. Write $\sigma^{-1}$ for
the reversed walk and $\sigma\eta$ for concatenation when the end
of $\sigma$ is the start of $\eta$.

A spur is a segment $(u,v,u)$. Inserting or deleting spurs gives
an equivalence relation on walks with fixed endpoints. Fix a vertex
$b$ of a connected graph $H$. The classes of closed walks based at
$b$, with concatenation as multiplication, form the fundamental
group $\Pi_1(H,b)$. The constant walk is the identity, and reversal
gives inverses. This definition requires no topology. For example,
choosing a spanning tree gives a free generating set: each edge
outside the tree contributes the closed walk that reaches that edge
along the tree, crosses it, and returns along the tree.

To form the reduced homotopy group, declare every closed walk of
length four trivial as well. More precisely, let $N$ be the normal
subgroup of $\Pi_1(H,b)$ generated by all classes
$[\lambda\sigma\lambda^{-1}]$, where $\sigma$ is a closed
four-walk and $\lambda$ runs from $b$ to its starting vertex. Set
\[
 \Pi_1^*(H,b)=\Pi_1(H,b)/N.
\]
We write $[\gamma]$ for the reduced class of a closed walk, and
usually omit the base point. In group identities, names of walks
stand for their classes, so spur cancellations are implicit.
Changing the base point along a walk
conjugates the corresponding classes, so properties such as
$[\gamma]^2=1$ do not depend on this choice. For a disconnected
graph we use the component containing the walk in question.

There is also a useful local description. In addition to spur moves,
one may replace $(u,v,w)$ by $(u,v',w)$ whenever both are walks.
Their union is the four-walk $(u,v,w,v',u)$, so this replacement
preserves the reduced class; conversely these replacements and spur
moves impose all the four-walk relations. For instance a four-walk
$\sigma=(u,v,w,z,u)$ becomes $(u,z,u,z,u)$ after one replacement,
and then disappears by spur cancellations.

Spur moves change the length by two, and the additional relations
have even length. Hence
\[
 \ell:\Pi_1^*(H,b)\longrightarrow\Z/2\Z,
 \qquad \ell([\gamma])=\text{length}(\gamma)\pmod2
\]
is a well-defined homomorphism. In particular an odd closed walk
has a nonidentity reduced class. A graph homomorphism induces a
homomorphism of reduced groups, since it takes spurs to spurs and
four-walks to four-walks, and preserves length parity.

\subsection{Paths, homotopies, and complexes}
We now describe the topological language used later in the proof.
The geometric realization of a graph replaces each edge by an
interval and joins its endpoints to the corresponding vertices.
A topological path is a continuous map $P:[0,1]\to Z$ into a
space $Z$. Two paths with the same endpoints are homotopic relative
to their endpoints if there is a continuous map
$h:[0,1]^2\to Z$ with
\[
 h(u,0)=P(u),\quad h(u,1)=Q(u),\quad
 h(0,s)=P(0),\quad h(1,s)=P(1).
\]
A closed path is null-homotopic if it is homotopic, with its base
point fixed, to the constant path. Equivalently, it extends to a
continuous map of a disk. The topological fundamental group is the
group of based closed paths modulo these homotopies. For the
geometric realization of a graph it agrees with the walk-and-spur
definition above.

A cell complex is a space assembled from points, intervals, disks,
and higher-dimensional balls by attaching each ball's boundary to
the cells already present. These pieces are its cells; a subcomplex
is a union of cells containing the boundaries of all its cells.
The usual term for such a space is a CW complex. We only use
complexes obtained by the explicit finite polygon gluings below,
their covering spaces, and rectangular decompositions of $\R^d$.

For a graph $H$, attach a disk along every closed four-walk, with
each of the four boundary edges traversing the indicated graph
edge. Call the resulting two-dimensional complex $H^\square$.
Boundary vertices or edges may be identified, since the walk may
repeat them. Attaching these disks makes precisely the four-walk
classes trivial, and thus
\[
 \Pi_1(H^\square,b)\cong\Pi_1^*(H,b).
\]
In particular, $[\gamma]=1$ means that the loop represented by
$\gamma$ bounds a disk in $H^\square$. This interpretation explains
why the combinatorial four-walk relations suffice when we extend
maps across disks later in the proof.

A rectangular CW decomposition of $\R^d$ partitions it into the
relative interiors of axis-aligned closed rectangles, including
rectangles of lower dimension. The boundary of every closed cell
is a union of lower-dimensional cells. A geometric face of a
rectangle may therefore consist of several cells. The decomposition
is locally finite if every compact set meets only finitely many
cells; it is regular if each closed cell is homeomorphic to a closed
ball with its stated boundary. Local finiteness ensures that maps
defined continuously on cells and agreeing on their boundaries
give a continuous map on all of $\R^d$.

The local shape of a rectangle at a point of one of its faces is
described by coordinate orthants: a coordinate at a lower endpoint
must increase to enter its interior, one at an upper endpoint
must decrease, and the other coordinates are unrestricted. We
call these local collections of directions its axis-orthant germs.

The mapping cylinder of a map $f:E\to B$ is obtained from
$B$ and $E\times[0,1]$ by identifying $(z,0)$ with $f(z)$.
Moving down each cylinder segment gives a deformation to $B$
which leaves $B$ fixed. This elementary model describes the
strips between adjacent levels in Section~\ref{sec:carriers}.

\subsection{Presentation complexes and covering spaces}
A group presentation
\[
 \langle s_1,\ldots,s_m\mid r_1,\ldots,r_k\rangle
\]
has a presentation complex formed from
one vertex, an oriented loop for each generator, and a disk for
each relation. The boundary of that disk follows the generator
loops in the order specified by the relation word; inverse letters
are traversed backwards. Its fundamental group is the presented
group, because its based paths are words in the generators, with
free cancellations and the stated relations.

The presentation complex used here is
\[
 Y=Y\bigl(\langle a,t\mid tat^{-1}=a^3\rangle\bigr).
\]
It has one vertex, two oriented loops, and one disk attached along
$tat^{-1}a^{-3}$. We use the same letters $a,t$ for the generators
of its fundamental group.

A covering map $p:\widetilde Z\to Z$ is locally a disjoint union
of copies of the same open subset of $Z$. A path in $Z$ has a unique
lift once its starting point over $Z$ has been chosen. A deck
transformation is a homeomorphism of $\widetilde Z$ commuting with
$p$. A universal cover is a connected covering space in which
every loop is null-homotopic.

For a presentation complex there is a concrete model of its
universal cover. Vertices are group elements $g$; the oriented
$s$-edge goes from $g$ to $gs$; and each relation disk has a lift
starting at each vertex and following the same relation word.
Deck transformations act by left multiplication on these vertices,
edges and disks. In particular we may name the vertices of $\tY$
by words in $a,t$, and write $a^sC$ for the translate of a
subcomplex $C\subseteq\tY$.

The homomorphism
\[
 \chi:\Pi_1(Y)\longrightarrow\Z/2\Z,
 \qquad \chi(a)=1,\quad \chi(t)=0,
\]
is well-defined, since the two sides of $tat^{-1}=a^3$ have the
same parity. Quotienting $\tY$ by the deck transformations in
$\ker\chi$ gives the double cover $\hY\to Y$. It has two
vertices, indexed by parity: every $a$-edge changes the sheet and
every $t$-edge stays in its sheet. There are two lifted relation
disks. Left multiplication by $a$ induces the deck involution
$\tau$ exchanging the two sheets. This cover is finite and hence
compact. A graph's bipartite double cover is the analogous graph
with vertices $v^+,v^-$ and an edge $v^+w^-$ whenever $vw$ is an
edge; its involution exchanges the signs.

An involution on a space is a map $\tau$ with $\tau^2$ equal to the
identity; it is free if it has no fixed point. A map between spaces
with involutions is equivariant if it commutes with the
involutions. These are the only equivariance requirements in the
map $\eta$ constructed in Section~\ref{sec:positive}.

\subsection{Metrics and controlled contractions}
A simplex is the convex hull of affinely independent points: an
interval, triangle, tetrahedron, and so on. A triangulation
subdivides a space into simplices whose intersections are common
faces. The finite complexes in the proof admit finite triangulations. A piecewise linear (PL) map is affine on every
simplex after a finite subdivision. An edge path follows finitely
many edges of a triangulation. A triangle move replaces one side
of a triangle by its other two sides, or reverses that operation;
it gives a homotopy inside the triangle with endpoints fixed. A
backtrack move inserts or deletes a path along an edge and back.
Give each simplex a Euclidean
metric agreeing on common faces; lengths of paths then define the
piecewise Euclidean path metric. We fix such a metric on $Y$ and
lift it to its covers, so deck transformations are isometries and
covering projections do not increase distances.

For a connected finite subcomplex $C$, its intrinsic path distance
$d_C(x,y)$ is the infimum of lengths of paths from $x$ to $y$
lying entirely in $C$. This can exceed the ambient distance, since
paths leaving $C$ might be shorter. If $C\subseteq C'$ then
$d_{C'}(x,y)\leq d_C(x,y)$ for $x,y\in C$; the inclusion is
therefore nonexpansive. On a rectangle or its boundary, the
rectilinear path distance is the infimum of the sums of coordinate
lengths of polygonal paths staying in that set. It is infinite
when no such path joins the two points. On a finite Euclidean
complex, a finite PL
map is Lipschitz for these path metrics: bound its affine constants
on its finitely many simplices, and add those bounds along a path.

A map $f$ is $L$-Lipschitz if
$d(f(x),f(y))\leq Ld(x,y)$ for all $x,y$; write $\Lip f$ for
its least such constant. For $F:\R^d\to(C,d_C)$, write $\Lip_iF\leq L_i'$ if
\[
 d_C(F(u+se_i),F(u+te_i))\leq L_i'|s-t|
 \quad(u\in\R^d,\ s,t\in\R).
\]
We use these separate coordinate bounds to choose different scales
in the first and the remaining coordinates. The uniform distance
between paths is
$d_\infty(P,Q)=\sup_{u\in[0,1]}d_C(P(u),Q(u))$.

A contraction of $C$ to $p$ is a continuous map
$H:C\times[0,1]\to C$ satisfying
$H(z,0)=z$ and $H(z,1)=p$. It fixes $p$ if $H(p,s)=p$ for every
$s$. Its time tracks have speed at most $D$ if
\[
 d_C(H(z,s),H(z,s'))\leq D|s-s'|
 \quad(z\in C,\ s,s'\in[0,1]).
\]
It is jointly $J$-Lipschitz if it is $J$-Lipschitz when
$C\times[0,1]$ has distance $d_C(z,z')+|s-s'|$.
The time-speed bound controls how far a point moves during the
contraction. The joint bound also controls how nearby starting
points move, and it may be much larger. Keeping these two bounds
separate is essential in Sections~\ref{sec:carriers} and
\ref{sec:extension}.

Finally, the probability simplex of a finite vertex set $V$ is
\[
 \Delta(V)=\{p:V\to[0,1]:\textstyle\sum_{v\in V}p(v)=1\}.
\]
Write $\delta_v$ for the mass concentrated at $v$ and
$\supp p=\{v:p(v)>0\}$. For $A\subseteq V$, its face
$\Delta(A)$ consists of the masses supported on $A$.
Two subsets of graph vertices are completely adjacent if every
vertex of one is adjacent to every vertex of the other. The space
$\calD(H)$ used in Section~\ref{sec:positive} is a union of
products $\Delta(A)\times\Delta(B)$ for such pairs. These are
convex Euclidean polytopes, and their intersections are faces, so
their union is a finite polyhedral complex. This concrete
description is all we need about that space. A fixed-shore fiber
means that one probability coordinate is held fixed; the allowed
values of the other coordinate form a simplex face, hence a
convex set.

\section{Clopen markers and rectangular orbit tiles}\label{sec:atlas}

The first step is to partition each orbit into rectangles of prescribed
coordinate scales. A clopen marker set supplies the centers; finite
clopen priorities make the choices continuous. We use rectangular CW
decompositions in the sense of Section~2, so a face of a tile may be
a union of smaller cells. This allows adjoining tiles to have different
subdivisions while their boundary maps still agree.

\begin{lemma}\label{lem:clopen}
Let a finite set of homeomorphisms generate a loopless graph on $X$, and
suppose the graph has maximum degree at most $\Delta$.  If $X$ has a
countable clopen cover by independent sets, then the graph has a clopen
maximal independent set and a continuous proper $(\Delta+1)$-coloring.
In particular this applies to any finite-displacement graph of a free
$\Z^d$-action.  For the graph generated by one homeomorphism, the absence
of fixed points suffices and gives a continuous proper three-coloring.
\end{lemma}

The proof is essentially the same proof of \cite[Proposition 4.6]{KST}, see, for example, \cite[Lemma 2.1, 2.2 and 2.3]{Bernshteyn}.

Fix positive integers $L_1,\ldots,L_d$ and use the scaled norm
\[
 \|u\|_L=\max_{1\leq i\leq d}\frac{|u_i|}{L_i}.
\]
An orbit is identified with $\Z^d$ by a choice of origin, and its ambient
real-coordinate space with $\R^d$.

\begin{lemma}[Rectangular atlas]\label{lem:atlas}
For each $d$ there are constants $c>0$, $C<\infty$ and $L_0$ with the
following property.  For every free $\Z^d$-action and integers
$L_i\geq L_0$, one can continuously and translation-covariantly assign
to each pointed orbit a rectangular CW decomposition of $\R^d$ such that:
\begin{enumerate}
\item every vertex has integer coordinates;
\item every positive side in coordinate $i$, of every cell, has length
between $cL_i$ and $CL_i$;
\item the numbers of tiles meeting a given tile, and of cells in its
boundary, are bounded in terms of $d$ alone.
\end{enumerate}
Continuity here means that every bounded part of the pointed atlas is
determined by finitely many continuous finite-valued orbit data.
\end{lemma}

\begin{proof}
Apply Lemma~\ref{lem:clopen} to the finite displacement set
$0<\|g\|_L\leq1$.  A clopen maximal independent set produces, in each
orbit, a marker set $A\subset\Z^d$ whose distinct points have scaled
distance greater than one, while every lattice point is within scaled
distance one of a marker.

Fix $S=30$.  Packing disjoint scaled half-unit boxes gives a bound
$P_d$ for the number of markers within distance $S$ of any marker; for
example $P_d=(2S+3)^d$ suffices.  Continuously color the marker proximity
graph at radius $S$ with at most $P_d$ colors.  Regard these colors as
priorities.  For every marker $a$, choose its lower and upper $i$-cuts
from the integer intervals
\begin{equation}\label{eq:cuts}
 [a_i-3L_i,a_i-2L_i],\qquad [a_i+2L_i,a_i+3L_i].
\end{equation}
Process priorities in increasing order.  Choose each cut at distance at
least
\[
 \delta_i=\left\lfloor\frac{L_i}{16P_d}\right\rfloor
\]
from all previously chosen parallel cuts belonging to markers within
distance $S$.  There are at most $2P_d$ earlier cuts.  They exclude at
most $4P_d\delta_i+2P_d$ integer candidates, fewer than $L_i+1$ when
$L_i$ is sufficiently large.  Thus a cut can always be chosen; use the
first allowed offset in a fixed ordering.  Equal-priority markers are
farther apart than $S$, so simultaneous choices do not interfere.
For $L_i\geq32P_d$, we have $\delta_i\geq L_i/(32P_d)$.

Let $Q_a$ be the resulting box at $a$.  These boxes cover $\R^d$:
rounding a real point to a lattice point and then choosing a covering
marker puts it within scaled distance at most $3/2$ of that marker,
whereas every half-width in \eqref{eq:cuts} is at least two.  Subdivide
each $Q_a$ by all coordinate cut levels of boxes meeting it.  Keep an
elementary full-dimensional box precisely when its interior is not
contained in a box of smaller priority.  Away from the cut hyperplanes,
every point has a unique least-priority containing box.  Membership in
each overlapping box is constant on an elementary interior.  Hence the
retained boxes have disjoint interiors and cover space by their closures.

If $Q_b$ meets $Q_a$, then $\|b-a\|_L\leq6$.  Thus all cut owners used
to subdivide $Q_a$ are within six of $a$, and any two are within twelve
of one another.  Their distinct parallel cuts are separated by
$\delta_i$.  Every retained side has length between $\delta_i$ and
$6L_i$.  Packing bounds the number of overlapping marker boxes and the
number of elementary boxes belonging to each marker.  It also bounds
the number of retained boxes meeting a given retained box.

It remains to refine common boundaries without subdividing the
full-dimensional interiors again.  For a point $x$ and a retained box
$R$ containing it, let $F_R(x)$ be the smallest closed face of $R$
containing $x$.  Put
\[
 Q(x)=\bigcap_{R\ni x}F_R(x).
\]
The point $x$ lies in the relative interior of this rectangle.  At every
$y\in\relint Q(x)$ it lies in the relative interior of the same faces
$F_R(x)$.  The corresponding axis-orthant germs cover a neighborhood,
as they do at $x$.  An additional incident tile at $y$ would therefore
have interior overlapping an already incident tile.  Consequently the
incident tiles and their minimal faces are constant on $\relint Q(x)$.
These relative interiors partition space.  On the boundary of $Q(x)$,
at least one minimal face becomes smaller.  Local finiteness now gives
a regular rectangular CW decomposition.

If a cell belongs to a tile with parent marker $a$, its endpoint levels
come from tiles meeting that tile.  Those tiles have parent markers
within six of $a$, and their cut owners are within twelve of $a$.
Two such cut owners are within twenty-four of one another, still less
than $S$.  The same separation bound therefore holds for every positive
side of every boundary cell.  The packing bounds also give a uniform
bound on their number.  We may take $C=6$ and $c=1/(32P_d)$.

All choices use finite neighborhoods of clopen marker data and finitely
many priority stages.  The choices of offsets and the subsequent
refinements commute with integer translation.  This proves continuity
and covariance as well as the stated geometric bounds.
\end{proof}

\section{Finite carriers with logarithmically short contractions}
\label{sec:carriers}

Use the presentation complex $Y$, its universal cover $\tY$, and
its parity double cover $\hY$ from Section~2. The relation
$tat^{-1}=a^3$ lets us reach distant vertices $a^n$ by paths of
logarithmic length. We need more than short paths: for each interval
we construct a finite subcomplex, called its carrier, and a continuous
family of short paths within it. Carriers for nested intervals must
be nested as well, so that maps on cell boundaries remain in the
carrier chosen for the containing cell. Distances within a carrier
always mean intrinsic path distances unless otherwise specified.

\begin{lemma}[Interval carriers]\label{lem:carriers}
Fix an integer $R\geq1$.  For each integer interval $I=[A,B]$ of length
at most $R$, there is a finite subcomplex $C_I\subset\tY$ with
distinguished vertices $a^n$, $n\in I\cap\Z$, such that:
\begin{enumerate}
\item $C_J\subseteq C_I$ when $J\subseteq I$, and
$a^sC_I=C_{I+s}$ for every integer $s$;
\item $C_I$ has a contraction $H_I$ to $a^A$, fixing $a^A$, whose time
tracks have speed at most $D(R)=O(\log(R+2))$;
\item the contractions are jointly Lipschitz and translation covariant.
At fixed $R$, their joint Lipschitz constants have a finite common bound
$J(R)$.
\end{enumerate}
No quantitative estimate on $J(R)$ is required.
\end{lemma}

\begin{proof}
Use the common height
\[
 K=2+\lceil\log_3(R+1)\rceil
\]
for all intervals.  The vertices of $C_I$ are
\[
 (r,j)=a^rt^j,\qquad r\in I\cap\Z,\quad 0\leq j\leq K.
\]
Include all vertical $t$-edges, all horizontal edges
$(r,j)\to(r+3^j,j)$ whose endpoints belong to $I$, and every relation
cell whose upper edge has endpoints $r,r+3^{j+1}$ in $I$.  The bottom
of this cell is the three-edge horizontal path at level $j$.
The affine representation $a(x)=x+1$, $t(x)=3x$ shows that the named
vertices are distinct.  The lifted edges and cells are therefore the
indicated subcomplex of the universal cover.  The construction gives
both inclusions and translation covariance literally.

\smallskip\noindent
\emph{Retraction to the bottom.}
Each horizontal level is a forest.  A strip between adjacent levels is
the mapping cylinder of the map from the upper forest to the lower
forest which replaces each upper edge by its three-edge lower path.
Choose a PL rectangle model for a relation cell with three edges on its
bottom side and one edge on each other side.  In cylinder coordinates
$(\theta,h)$ the downward homotopy is
$h\mapsto\max(h-s,0)$.  These maps agree on vertical edges and fix the
lower graph.  Each strip has a fixed time-speed bound.  Retract the
strips successively, from top to bottom, in $K$ equal time intervals.
This is a finite PL deformation to the bottom interval, fixes that
interval, and has time speed $O(K)$.  Its point-variable Lipschitz
constant may be large, but is finite.

\smallskip\noindent
\emph{Compressed paths and carries.}
Translate so that $A=0$.  If $n=\sum_{j=0}^k r_j3^j$ is the base-three
expansion of $n>0$, put
\begin{equation}\label{eq:compressed}
 p_n=t^ka^{r_k}t^{-1}a^{r_{k-1}}\cdots t^{-1}a^{r_0},
\end{equation}
and let $p_0$ be constant.  This path runs from $0$ to $n$, has length
$O(K)$, and has every horizontal coordinate in $[0,n]$.  It lies in
$C_{[0,n]}$.

Ordinary base-three carrying changes $p_na$ to $p_{n+1}$: replace a
terminal $a^3$ by $tat^{-1}$, cancel an adjacent $t^{-1}t$, and continue
upward when required.  At most $K$ carries occur.  Every relation cell
used lies above an increasing horizontal interval in $[0,n+1]$ and at
height at most $K$.  The intermediate edge words have length $O(K)$.
Triangulate each relation polygon by a central vertex.  Crossing one
such polygon is a finite sequence of triangle moves and changes the
length bound only by a fixed constant.  Hence one integer
$M_0=O(K)$ bounds the number of edges in every intermediate path, for
all $0\leq n<B$ simultaneously.

\smallskip\noindent
\emph{A fixed-slot path homotopy.}
We spell out how to preserve parametrizations.  If two edge paths are
related by triangle moves and backtrack moves, with at most $M_0$
edges in every intermediate path, divide the path-parameter interval
$[0,1]$ into $T=M_0+2$ slots of width $1/T$.  Perform the endpoint-fixed
homotopies in the first $T-1$ slots, keeping the final slot constant
until the moving-endpoint correction below.  Each edge occupies one
slot; unused slots are pauses at the endpoint.  The first $T-1$ slots
always contain a spare constant slot.  A pause can be moved
past an edge using
\[
 (A,A,B)\longleftrightarrow(A,B,B).
\]
A triangle insertion and a backtrack deletion take the forms
\[
 (A,A,C)\longleftrightarrow(A,B,C),\qquad
 (A,B,A)\longleftrightarrow(A,A,A).
\]
Move spare slots to the location of a move and back afterward.  There
is always a spare slot because every path uses at most $M_0$ edges.
In every time-slot square between consecutive homotopy rows, the four
corner images lie in a single target simplex.  Triangulate the square
and extend affinely.  On either of its triangles, the derivative in
the path-time direction is the corresponding row-edge vector divided
by the slot width.  It is at most $\Delta T$, where $\Delta$ is a fixed
bound on the diameters of target simplices.  The number of homotopy
rows affects the other derivative, but not this path-time bound.

Parametrize each $p_n$ canonically using $T$ slots of width $1/T$,
traversing its edges first and pausing thereafter; call the resulting
path $P_n$.  Applying the preceding construction to the carry homotopy,
using only the first $T-1$ slots, gives a finite PL endpoint-fixed
homotopy $h_n$ from $p_na$ to $p_{n+1}$ in these exact canonical slots.
Its path-time speed is $O(T)$.  Reserve the last slot for the following
moving-endpoint correction.

\smallskip\noindent
\emph{Interpolation between integers.}
For $u=n+s$, $0\leq s\leq1/2$, follow the nonconstant slots of $P_n$,
move to $n+2s$ in the next slot, pause there until the final slot, and
return to $n+s$ in the final slot.  Formally specify these values at
the slot endpoints and triangulate each parameter rectangle.  Its
corner images lie in a single bottom edge, or in one unchanged edge
of $p_n$.  This gives a finite PL family with speed $O(T)$.

For $1/2\leq s\leq1$, use $h_n(2s-1)$ in the first $T-1$ slots and
return from $n+1$ to $n+s$ in the final slot.  Triangulate the final
slot using the finitely many row parameters of $h_n$; all its images
lie in the bottom edge.  At $s=1/2$ both definitions are exactly the
canonical $p_na$ followed by the return to $n+1/2$.  At $s=0$ and $s=1$
they are exactly $P_n$ and $P_{n+1}$.  Thus adjacent unit intervals
glue literally, not merely up to reparametrization.

We obtain a finite PL family $P_u$, from $0$ to every bottom point
$u\in[0,B]$, with speed $O(T)=O(K)$.  The path $P_0$ is constant.
Reverse this family to contract the bottom interval to zero.  Concatenate
it with the downward retraction, using half the time for each stage.
Composition with the final downward projection does not change the
time-speed bound of the bottom stage.  The resulting contraction fixes
zero and has track speed $O(K)$.

Translate this construction for a general left endpoint $A$.  At
fixed $R$ there are only finitely many interval lengths.  The chosen
finite PL maps are jointly Lipschitz in their intrinsic finite-complex
metrics, so their finitely many constants have a maximum $J(R)$.
The common height and common slot budget preserve the asserted
covariance and uniform track bound.
\end{proof}

\begin{remark}
The contractions for overlapping intervals need not agree.  What is
essential is the literal inclusion $C_J\subseteq C_I$ at one common
height.  Intrinsic distances can only decrease under this inclusion,
so estimates for a boundary cell pass to a containing cell's carrier.
\end{remark}

\section{Relative extension with one controlled coordinate}
\label{sec:extension}

We now extend maps from cell boundaries to cell interiors. The
first coordinate follows the short paths supplied by the carriers;
the other coordinates interpolate between those paths. The key is
to control the first-coordinate speed without multiplying it by
the potentially large joint Lipschitz constant of a contraction.

\begin{lemma}[A relative path contraction]\label{lem:path}
Let a metric space $C$ have a jointly $J$-Lipschitz contraction $H$ to
$p$, fixing $p$, whose time tracks have speed at most $D$.  Fix $q\in C$
and set $Q(u)=H(q,1-u)$.  There is an endpoint-preserving deformation
$\calR(P,s)$ of each $S$-Lipschitz path $P:p\to q$ to $Q$ such that
\[
 \calR(P,0)=P,\quad \calR(P,1)=Q,\quad
 \Lip_u\calR(P,s)\leq3\max(S,D).
\]
The starting parametrization is unchanged.  For an absolute constant
$K_0$, and for input paths with the same bound $S$,
\begin{equation}\label{eq:path-family}
 \begin{aligned}
 &d_\infty\bigl(\calR(P,s),\calR(P',s')\bigr)\\
 &\hspace{1cm}\leq K_0(1+J)
 \bigl(d_\infty(P,P')+(S+D)|s-s'|\bigr).
 \end{aligned}
\end{equation}
\end{lemma}

\begin{proof}
Use four stages, each on one quarter of the deformation parameter.
First compress $P$ from the whole time interval into its first third,
leaving an ending pause.  Next keep $P$ in the first third and grow
the spur $H(q,[0,r])$ followed by its reverse in the other two thirds,
as $r$ increases from zero to one.  In the third stage decrease $r$
from one to zero, putting the following paths in three equal slots:
\begin{equation}\label{eq:three-slots}
 P(ru),\qquad H(P(r),u),\qquad Q(u),\qquad 0\leq u\leq1.
\end{equation}
All junctions match.  At $r=0$ the first two slots are constant at $p$,
because $H$ fixes $p$.  Finally expand the last copy of $Q$ to occupy
the whole interval, removing the initial pause.

In the first and last stages every occupied interval has length at
least one third.  In \eqref{eq:three-slots}, the speeds are bounded
by $3S$, $3D$, and $3D$, respectively.  The spur stage has the same
bound.  This proves the path-time estimate.

All parameter changes in these formulas have universal Lipschitz
bounds.  For example, changing $r$ changes $P(r)$ by at most
$S|\Delta r|$, and subsequent evaluation under $H$ costs at most a
factor $J$.  Comparing input paths in the uniform metric and applying
these estimates on each stage proves \eqref{eq:path-family}; estimates
across stage boundaries follow by concatenation.  Notice that $J$
does not enter the bound in the path-time direction.
\end{proof}

\begin{lemma}[Anisotropic extension]\label{lem:extension}
Use the atlas of Lemma~\ref{lem:atlas} with $L_1=L$ and $L_i=M$ for
$i>1$, and the carriers of Lemma~\ref{lem:carriers} with $R=CL$.
There is a continuous map $F_x:\R^d\to\tY$, canonically determined
by the pointed atlas, such that a cell with first-coordinate interval
$I$ maps into $C_I$, and a cell with constant first coordinate $n$
maps constantly to $a^n$.  It satisfies
\begin{equation}\label{eq:coordinate-bounds}
 \Lip_1 F_x\leq A_d\frac{\log(L+2)}L,\qquad
 \Lip_i F_x\leq\frac{B_d(L)}M\quad(i>1),
\end{equation}
where $A_d$ depends only on $d$ and $B_d(L)<\infty$ is independent of
$M$ and $x$.  The maps are translation covariant, and
$x\mapsto F_x(0)$ is continuous and finite-valued.
\end{lemma}

\begin{proof}
All estimates in the induction use the intrinsic metric of the
appropriate carrier.  Inclusions of interval carriers are
nonexpansive for these metrics.  Put $D=D(CL)$ and $J=J(CL)$.
Map every cell with constant first coordinate $n$ constantly to $a^n$.
On a one-dimensional cell $I=[A,B]$ varying in the first coordinate,
use $Q_I(u)=H_I(a^B,1-u)$, affinely parametrized over $I$.
Its first-coordinate speed is at most $D/(cL)$.

Suppose the map has been defined on all cells of dimension below $k$.
On such cells suppose its physical coordinate bounds are
$A_{k-1}D/L$ and $B_{k-1}/M$.  Consider a $k$-cell $I\times T$, where
$I=[A,B]$ varies in coordinate one and $T$ is a transverse rectangle.
The end faces are constant at $a^A$ and $a^B$.  The remaining boundary
defines paths
\[
 P_z(u)=F_x(A+(B-A)u,z),\qquad z\in\partial T.
\]
All these paths lie in $C_I$.  A first-coordinate boundary line may
pass through several smaller cells, but the bounds add over their
physical lengths.  Consequently
\[
 \Lip_u P_z\leq CA_{k-1}D=:S.
\]
Similarly, the boundary family is $B_{k-1}/M$-Lipschitz in the intrinsic
rectilinear path metric of $\partial T$, with the uniform path metric
in the range.  This assertion uses carrier inclusions and does not
require the different contractions to agree.

Let $o$ be the center of $T$ and let $\ell_j$ be its positive side
lengths.  Set
\[
 \rho(z)=\max_j\frac{2|z_j-o_j|}{\ell_j},\qquad
 \pi(z)=o+\frac{z-o}{\rho(z)}\quad(\rho(z)>0).
\]
For $\rho(z)\geq1/2$ define
\begin{equation}\label{eq:collar-extension}
 F_x(A+(B-A)u,z)
 =\calR\bigl(P_{\pi(z)},2(1-\rho(z))\bigr)(u).
\end{equation}
For $\rho(z)\leq1/2$ use $Q_I(u)$.  At the outer boundary the
deformation parameter is zero, so the prescribed boundary map is
preserved exactly.  At the inner boundary it is one.  Thus the inner
half-box is a constant path family and the radial singularity never
occurs.

The first-coordinate bound in this extension is
\[
 \frac{3\max(CA_{k-1},1)D}{cL}.
\]
We may therefore take $A_1=1/c$ and
\[
 A_k=\frac{3\max(CA_{k-1},1)}c,
\]
enlarging the sequence to be nondecreasing if necessary.

Every transverse length lies between $cM$ and $CM$.  On the outer
half-box, $\pi$ has a piecewise Lipschitz constant bounded in terms of
$c,C,d$, and $2(1-\rho)$ has coordinate Lipschitz constants at most
$4/(cM)$.  The image under $\pi$ of a coordinate segment has boundary
path length bounded by a fixed multiple of its length.  Applying
\eqref{eq:path-family} gives, after increasing a dimension-dependent
constant $K_d$,
\begin{equation}\label{eq:transverse-recurrence}
 B_k\leq K_d(1+J)
 \bigl(B_{k-1}+(CA_{k-1}+1)D\bigr),\qquad B_1=0.
\end{equation}
For a one-dimensional $T$, its two outer intervals project to their
respective endpoints; the inner interval is constant, so the same
estimate applies.  In all dimensions one splits coordinate segments
at the inner box and at the finitely many radial-face changes.
Thus $B_k$ is finite after $L$ is fixed and is independent of $M$.
The first physical coordinate is never mixed with a transverse
coordinate in \eqref{eq:collar-extension}.

The maps agree on every shared boundary.  Local finiteness gives a
global continuous map.  A compact coordinate segment is partitioned
into finitely many cell pieces, and summing the coordinate estimates
over their lengths proves the global bounds.  Since
$D(CL)=O(\log(L+2))$, these are \eqref{eq:coordinate-bounds}.

The rule on each cell is determined by its rectangle, its boundary
map and the specified interval contraction.  Induction therefore
proves translation covariance.  At fixed $L,M$, the cells used to
evaluate $F_x(0)$ and all their boundary recursions lie in a bounded
coordinate box.  Integer cuts and finite marker priorities have only
finitely many patterns there.  Each pattern occurs on a clopen subset
of $X$, and the construction is deterministic on it.  Hence
$x\mapsto F_x(0)$ is continuous and finite-valued.  This statement
does not require a uniform binary coding radius on the noncompact
free shift.
\end{proof}

\section{Approximate parity fields and finite graph targets}
\label{sec:positive}

The geometric construction first gives an approximate parity map:
each generator nearly exchanges the two sheets of $\hY$. We then
map the sheets to pairs of probability distributions on $H$. Their
support condition turns the approximation into an exact graph
homomorphism by choosing a vertex of maximum mass.

\begin{proposition}\label{prop:parity}
For every free zero-dimensional Polish $\Z^d$-action and every
$\eps>0$, there is a finite-valued continuous map $f:X\to\hY$ such that
\[
 \dist\bigl(f(e_i\cdot x),\tau f(x)\bigr)<\eps
 \qquad(x\in X,\ 1\leq i\leq d).
\]
\end{proposition}

\begin{proof}
Use the unimodular basis
\[
 v_1=e_1,\qquad v_j=e_j-e_1\quad(2\leq j\leq d).
\]
The first coordinate of a group element in this basis is the sum of
its standard coordinates.  In particular every original generator
has first coordinate one, and $e_j=v_1+v_j$ for $j>1$.
Apply Lemma~\ref{lem:extension} in these coordinates.  Changing the
orbit origin by $g$ translates the atlas by $-g$, so covariance is
\begin{equation}\label{eq:covariance}
 F_{g\cdot x}(u-g)=a^{-g_1}F_x(u).
\end{equation}
Project $F_x(0)$ to $\hY$ and call the result $f(x)$.  Projection is
nonexpansive and left multiplication by $a$ induces $\tau$.  Evaluate
\eqref{eq:covariance} at $u=g$ and use a first-coordinate segment,
followed when needed by one transverse segment.  The required error
is at most
\[
 A_d\frac{\log(L+2)}L+\frac{B_d(L)}M.
\]
Choose $L$ first to make the first term small, and then choose $M$.
Continuity and finite-valuedness follow from
Lemma~\ref{lem:extension}.
\end{proof}

Recall the probability simplex and complete adjacency of supports
from Section~2. For a finite graph $H$, define
\[
 \calD(H)=\bigl\{(p,q)\in\Delta(V(H))^2:
     \supp p\text{ is completely adjacent to }\supp q\bigr\}.
\]
This is a finite polyhedral complex, a union of products of simplex
faces.  The involution $\tau_H(p,q)=(q,p)$ is free when $H$ is loopless.

\begin{lemma}\label{lem:hom-complex}
If $H$ contains an odd closed walk $\gamma$ with $[\gamma]^2=1$ in
$\Pi_1^*(H)$, there is an equivariant continuous map
$\eta:\hY\to\calD(H)$.
\end{lemma}

\begin{proof}
Work first in the connected component containing $\gamma$.  Choose a
neighbor $s(v)$ of each vertex.  Map its bipartite double cover to
$\calD(H)$ by
\[
 v^+\longmapsto(\delta_v,\delta_{s(v)}),\qquad
 v^-\longmapsto(\delta_{s(v)},\delta_v).
\]
The edge from $v^+$ to $w^-$ maps to the two straight segments through
$(\delta_v,\delta_w)$.  These lie in convex fixed-shore fibers, and
the construction is equivariant under swapping the two shores.

A lifted four-walk $v^+,w^-,x^+,y^-,v^+$ maps to a null-homotopic
loop.  Indeed, straighten its auxiliary-neighbor detours in the
fixed-shore fibers.  The resulting loop bounds
\[
 \Delta(\{v,x\})\times\Delta(\{w,y\})\subseteq\calD(H).
\]
Repeated vertices just give degenerate rectangles.  Spurs also
contract.  The hypothesis on $[\gamma]^2$ therefore implies that the
lift of $\gamma^2$ maps to a null-homotopic loop.

Write $\alpha:p\to\tau_Hp$ for the image of the lift of $\gamma$.
The loop $\alpha*(\tau_H\alpha)$ is null-homotopic.  Map the two
vertices of $\hY$ to $p,\tau_Hp$, its two positive $a$-edges to
$\alpha,\tau_H\alpha$, and its two $t$-loops constantly.  The
relation $tat^{-1}a^{-3}$, based at the first sheet, maps to
\[
 \alpha*\operatorname{reverse}
       (\alpha*(\tau_H\alpha)*\alpha).
\]
After cancelling the first spur, this is the inverse of
$\alpha*(\tau_H\alpha)$.  Fill that relation cell and define the
filling of the other cell by the involution.  This gives the required
equivariant map; inclusion of the chosen component into $H$ completes
the construction.
\end{proof}

\begin{proof}[Proof of Theorem~\ref{intro:positive}]
Put $n=|V(H)|$.  By uniform continuity of $\eta$ on the finite complex
$\hY$, choose the error in Proposition~\ref{prop:parity} small enough
that the composite $\eta f(x)=(p_x,q_x)$ satisfies
\[
 \|p_{e_i\cdot x}-q_x\|_\infty<\frac1{2n}.
\]
Choose $c(x)$ to maximize $p_x(v)$, breaking ties by a fixed ordering
of $V(H)$.  The map $\eta f$ is continuous and finite-valued, so this
selection is continuous.  Also $p_x(c(x))\geq1/n$.  If $y=e_i\cdot x$,
then
\[
 p_x(c(x))>0,\qquad
 q_x(c(y))\geq p_y(c(y))-\|p_y-q_x\|_\infty>0.
\]
The defining support condition of $\calD(H)$ shows that $c(x)$ and
$c(y)$ are adjacent.  Thus $c$ is the asserted continuous homomorphism.
\end{proof}

\section{Two further remarks}\label{sec:remarks}

\subsection{The converse of the positive theorem}
\label{subsec:converse}

The condition in Theorem~\ref{intro:positive} can be weakened: the
class of the odd closed walk need only have finite order.  In dimension
two this is \cite[Theorem~3.5.1]{GJKS}.  The construction above gives the
same conclusion in every finite dimension.

\begin{proposition}\label{prop:finite-order}
Let $H$ be a finite simple graph containing an odd closed walk $\gamma$
whose class in $\Pi_1^*(H)$ has finite order.  For every $d\geq1$, every
free continuous $\Z^d$-action on a zero-dimensional Polish space admits
a continuous homomorphism from its standard Schreier graph to $H$.
\end{proposition}

\begin{proof}
Let $N$ be the order of $[\gamma]$.  Since length parity takes this
class to $1\in\Z/2\Z$, the integer $N$ is even.  Put $m=N+1$ and replace
the auxiliary complex $Y$ by the presentation complex
\[
 Y_m=Y(\langle a,t\mid tat^{-1}=a^m\rangle).
\]
The assignment $a\mapsto1$, $t\mapsto0$ defines a double cover
$\widehat Y_m\to Y_m$, because $m$ is odd.  Write $\tau$ for its deck
involution.  The integer $m$ is fixed throughout this proof.

We first explain why Proposition~\ref{prop:parity} holds with this
double cover.  In Lemma~\ref{lem:carriers}, use the common height
\[
 K_m=2+\lceil\log_m(R+1)\rceil.
\]
The carrier vertices are again $a^rt^j$, with $r\in I\cap\Z$ and
$0\leq j\leq K_m$.  A horizontal edge at height $j$ now changes $r$ by
$m^j$, and the lower side of each relation cell consists of $m$ edges.
The affine maps $a(x)=x+1$ and $t(x)=mx$ show that the named vertices
are distinct.  Both the carrier inclusions and their translation
covariance are unchanged.  Each strip is still a mapping cylinder,
so the downward retraction has time-track speed $O_m(K_m)$.

For the second stage of the contraction, use base $m$.  If
$n=\sum_{j=0}^k r_jm^j$, where $0\leq r_j<m$, the path
\[
 p_n=t^ka^{r_k}t^{-1}a^{r_{k-1}}\cdots t^{-1}a^{r_0}
\]
stays over $[0,n]$ and has $O_m(K_m)$ edges.  Base-$m$ carrying changes
$p_na$ to $p_{n+1}$ through at most $K_m+1$ relation moves.  All
intermediate paths still have $O_m(K_m)$ edges.  Each relation polygon
has a fixed finite triangulation; hence the fixed-slot homotopies and
the interpolation between integers in Lemma~\ref{lem:carriers} apply
with the same parametrization conventions.  The resulting contractions have
\[
 D_m(R)=O_m(\log(R+2)),\qquad J_m(R)<\infty.
\]
The relative path contraction and anisotropic extension require only
these estimates and the carrier inclusions.  Consequently, for every
$\varepsilon>0$, they give a continuous finite-valued map
$f:X\to\widehat Y_m$ with
\[
 \dist\bigl(f(e_i\cdot x),\tau f(x)\bigr)<\varepsilon
 \qquad(1\leq i\leq d).
\]

We next construct an equivariant map
$\eta:\widehat Y_m\to\calD(H)$.  Use the map from the bipartite double
cover of $H$ described in Lemma~\ref{lem:hom-complex}.  Its extension
over the lifted four-walks sends every null-homotopic loop in the square
complex to a null-homotopic loop in $\calD(H)$.  Let
$\alpha:p\to\tau_Hp$ be the image of the lift of $\gamma$, and put
$\beta=\tau_H\alpha$.  The identity $[\gamma]^N=1$ therefore gives a
null-homotopy of
\[
 (\alpha*\beta)^{N/2}.
\]
Map the two vertices of $\widehat Y_m$ to $p,\tau_Hp$, its two
positive $a$-edges to $\alpha,\beta$, and its two $t$-loops constantly.
The relation cell based at $p$ has boundary image
\[
 \alpha*\operatorname{reverse}
       \bigl((\alpha*\beta)^{N/2}*\alpha\bigr).
\]
Cancelling the initial backtracking path leaves the reverse of
$(\alpha*\beta)^{N/2}$.  Fill this cell using its null-homotopy, and
fill the other cell by applying $\tau_H$.  This defines $\eta$.

Finally choose $\varepsilon$ using the uniform continuity of $\eta$
and apply the probability-coordinate rounding argument from the proof
of Theorem~\ref{intro:positive}.  The resulting vertex map is
continuous and takes adjacent points to adjacent vertices of $H$.
\end{proof}

\begin{remark}\label{rem:converse}
The converse of Theorem~\ref{intro:positive} fails even when its
conclusion is required for every dimension and every free continuous
action.  To see this, define a graph $H$ by the following array:
\[
\begin{array}{rrrrrr}
0&1&2&3&4&0\\
4&5&6&7&8&1\\
3&9&10&11&12&2\\
2&13&14&15&16&3\\
1&17&18&19&20&4\\
0&4&3&2&1&0.
\end{array}
\]
Equal labels denote the same vertex, and two vertices are adjacent
exactly when their labels occur horizontally or vertically next to
each other.  We will prove
\[
 \Pi_1^*(H)\cong\Z/4\Z,
\]
with length parity corresponding to reduction modulo two.

First fill the twenty-five displayed unit squares, obtaining a
complex $K$.  The sixteen interior vertices remain distinct.  Reading
the perimeter clockwise, beginning along the top edge, traverses
\[
 \gamma=(0,1,2,3,4,0)
\]
four times.  Thus $K$ is a disk attached to this five-cycle along
$\gamma^4$, with the displayed grid providing a subdivision.  Its
fundamental group has the presentation
\[
 \Pi_1(K)=\langle z\mid z^4=1\rangle,
 \qquad z=[\gamma].
\]

We must also consider four-walks that are not boundaries of displayed
faces.  Identify the interior vertices with their coordinates
$(i,j)$, $1\leq i,j\leq4$.  Their edges form the ordinary
$4$-by-$4$ grid.  The four interior corners
\[
 U=\{(1,1),(4,1),(1,4),(4,4)\}=\{5,8,17,20\}
\]
each have boundary neighbors $1,4$.  Every other interior perimeter
vertex has one boundary neighbor, and the remaining interior vertices
have none.  Vertex $0$ has no interior neighbor.

Classify a simple four-cycle by the number of its interior vertices.
With none it would lie in the boundary five-cycle, which has no
four-cycle.  With one, that vertex belongs to $U$ and the other three
are $1,0,4$; these are the four corner faces.  With two consecutive
interior vertices, their edge lies on the perimeter of the interior
grid, and their boundary neighbors determine one of the twelve side
faces.  At a corner, only the boundary neighbor on that same side is
adjacent to the boundary neighbor of the other endpoint.  With two
opposite interior vertices, both belong to $U$, giving exactly the
six additional cycles
\[
 (1,u,4,v,1),\qquad \{u,v\}\in\binom{U}{2}.
\]

Three interior vertices are impossible: the two adjacent to the lone
boundary vertex would have to be joined by a two-edge interior path.
For boundary vertices $1,4$, their interior neighbors are the four
corners, at pairwise grid distance at least three.  The interior
neighbors of $2$ and $3$ are respectively
\[
\begin{aligned}
 &\{(2,1),(4,2),(3,4),(1,3)\},\\
 &\{(3,1),(4,3),(2,4),(1,2)\}.
\end{aligned}
\]
Distinct points in either set also have grid distance at least three.
Finally, a four-cycle with four interior vertices is one of the nine
interior unit squares.  This proves that the only additional simple
four-cycles are the six listed above.

For each $u\in U$, its corner square makes the path $(1,u,4)$
homotopic, with endpoints fixed, to $(1,0,4)$.  Each additional cycle
is therefore already null-homotopic in $K$.  A four-walk with a
repeated vertex contains a backtrack and reduces to the constant
walk.  Filling all four-walks consequently adds no further relation,
and proves the asserted computation of $\Pi_1^*(H)$.

The generator $z$ is represented by the odd walk $\gamma$, so the odd
classes are exactly $z,z^3$.  Both have square $z^2\ne1$.  The same
conclusion holds for walks based elsewhere, by transporting them to
the chosen basepoint.  Hence $H$ contains no odd closed walk satisfying
the hypothesis of Theorem~\ref{intro:positive}.  Nevertheless $z$ has
finite order four, so Proposition~\ref{prop:finite-order}, using
$m=5$, gives the conclusion of that theorem in every dimension.  In
particular, $H\in\mathcal H_d$ for every $d\geq1$.
\end{remark}

\subsection{Strict dependence on the dimension}\label{subsec:strict-dimensions}

The classes of targets distinguish every two successive dimensions:
\[
 \mathcal H_1\supsetneq\mathcal H_2\supsetneq
 \mathcal H_3\supsetneq\cdots.
\]
We give a structural proof.  The finite targets will record small windows
of separated marker sets.  Their edge directions have two elementary
properties which prevent a homomorphism from using more lattice
directions than the target records.

For $m\geq1$, put
\[
 S_m=\{\pm e_1,\ldots,\pm e_m\}\subseteq\Z^m.
\]
An \emph{edge-direction assignment} on a graph $H$ is a map $\omega$
from oriented edges to $S_m$ with $\omega(v,u)=-\omega(u,v)$.
We will use the following two properties:
\begin{enumerate}
\item the sum of $\omega$ along every closed four-walk is zero,
including four-walks with repeated vertices;
\item if two walks $(u,v,w)$ and $(u,v',w)$ have all four edge
directions equal to the same $g\in S_m$, then $v=v'$.
\end{enumerate}
The second property says that a straight two-step walk has a unique
middle vertex once its endpoints and direction are fixed.

\begin{lemma}\label{lem:strict-marker}
For every $d\geq1$ and $L\geq4$, there is a finite loopless graph
$M_{d,L}\in\mathcal H_d$ with an edge-direction assignment in $S_d$
satisfying both properties above.
\end{lemma}

\begin{proof}
Put
\[
 Q=[-2L,2L]^d\cap\Z^d.
\]
The vertices of $M_{d,L}$ are the subsets $D\subseteq Q$ whose
distinct points have sup-distance greater than $L$ and which cover
$[-L,L]^d\cap\Z^d$ within sup-distance $L$.  Declare $D,E$ adjacent
when, for some $g\in S_d$,
\begin{equation}\label{eq:strict-marker-edge}
 D\cap(g+Q)=(g+E)\cap Q.
\end{equation}
This relation is symmetric.

For a free continuous $\Z^d$-action, use
Lemma~\ref{lem:clopen} to choose a clopen maximal independent set $A$
for the finite-displacement graph
$0<\|g\|_\infty\leq L$.  Each orbit marker set is $L$-separated and
$L$-covering.  Therefore
\[
 D_x=\{u\in Q:u\cdot x\in A\}
\]
is a vertex of $M_{d,L}$: a marker covering a point of
$[-L,L]^d$ necessarily lies in $Q$.  These windows depend continuously
on $x$.  Our action convention gives
$D_{g\cdot x}=D_x-g$ wherever both windows are recorded, so
\eqref{eq:strict-marker-edge} holds for $D_x,D_{g\cdot x}$ with
witness $g$.

Every oriented edge has a unique witness.  Choose $a\in D$ with
$\|a\|_\infty\leq L$, using coverage at zero.  If $g,h$ both witness
an edge $D\to E$, then $a-g,a-h\in E$.  These points lie in $Q$ and
are at distance at most two, so separation forces $g=h$.  Set
$\omega(D,E)=g$.  A loop would similarly put $a,a-g$ in the same
pattern at distance one, so the graph is loopless.

Consider a closed four-walk with directions $g_1,\ldots,g_4$ and
partial sums $s_j=g_1+\cdots+g_j$.  Starting with the chosen $a$ in
its initial pattern, the overlap identities successively record
$a-s_j$ in the $j$th pattern.  All these points lie in $Q$, since
$\|a-s_j\|_\infty\leq L+4\leq2L$.  At the end the initial pattern
contains both $a$ and $a-s_4$.  Their distance is at most $4\leq L$,
so they coincide.  This proves the first property.

Finally, suppose $D\to E\to D'$ has both directions equal to $g$.
In the common coordinate system its windows are
$Q,g+Q,2g+Q$.  Since $g$ is a signed standard generator,
\[
 g+Q\subseteq Q\cup(2g+Q).
\]
The part of $g+E$ in $Q$ is determined by $D$, and its part in
$2g+Q$ is determined by $2g+D'$.  Thus the endpoints determine
the entire middle pattern.  This proves the second property.
\end{proof}

We next recall the two-dimensional finite graphs needed in the
obstruction.  Write
\[
 B_{p,q}=\Gamma_{1,p,q},
 \qquad p,q\geq5,\quad \gcd(p,q)=1.
\]
These are the width-one twelve-tile graphs of GJKS.  The following
description also fixes their edge labels.  Begin with four based
oriented cycles $a,b,c,d$, with
\[
 |a|=|c|=p,\qquad |b|=|d|=q.
\]
Their basepoint is common; other vertices retain their cycle letter
and position. Words concatenate the cycles. Attach twelve grids
with the boundary words in the following table:
\[
\begin{array}{c|c|c|c}
 \text{top}&\text{right}&\text{bottom}&\text{left}\\ \hline
 c&a&c&a\\
 c&b&c&b\\
 d&a&d&a\\
 d&b&d&b\\
 dc&a&cd&a\\
 c&ba&c&ab\\
 cd&a&dc&a\\
 c&ab&c&ba\\
 c^q&a&d^p&a\\
 d^p&a&c^q&a\\
 c&b^p&c&a^q\\
 c&a^q&c&b^p
\end{array}
\]
Horizontal sides are read from left to right and vertical sides
from top to bottom.  Identify the corresponding boundary vertices,
keeping every tile interior private.  Label positive horizontal
edges by $e_1$ and positive vertical edges by $e_2$; reversing an edge
negates its label.  These labels agree under the identifications; write $\lambda$
for the resulting edge-label assignment.
The first four tiles include, in particular, the tori
$C_p\square C_q$ and $C_q\square C_p$, where $\square$ denotes
Cartesian graph product.

Put $X_2=\Free(2^{\Z^2})$.  The GJKS results we need are:
\begin{enumerate}
\item for every such $p,q$, there is a continuous direction-preserving
map $\theta_{p,q}:X_2\to B_{p,q}$
\cite[Theorem~2.5.1]{GJKS};
\item if $J\in\mathcal H_2$, then $B_{p,q}\to J$ for all sufficiently
large $p,q$
\cite[Corollary~3.3.1]{GJKS}.
\end{enumerate}
Here the coordinate labels are chosen so that direction preservation
means
\[
 \lambda\bigl(\theta_{p,q}(x),\theta_{p,q}(e_i\cdot x)\bigr)=e_i
 \qquad(i=1,2).
\]
In the second statement the source width is fixed at one.
Consequently all parameters used below can be chosen as distinct
sufficiently large odd primes.

\begin{lemma}[Rigidity of products of twelve-tile graphs]
\label{lem:strict-product}
Let $H$ be a finite loopless graph with an edge-direction assignment
in $S_m$ satisfying the two properties above.  Let
\[
 B=B_{p_1,q_1}\square\cdots\square B_{p_k,q_k},
\]
where all $2k$ parameters are distinct odd primes.  Give the source
edges their standard labels $\lambda\in S_{2k}$.
For every homomorphism $f:B\to H$, there is an $m\times2k$ matrix $M$
whose columns are signed unit vectors using pairwise distinct target
axes, such that
\begin{equation}\label{eq:strict-rigidity}
 \omega\bigl(f(b),f(b')\bigr)=M\lambda(b,b')
\end{equation}
on every edge.  In particular, $2k\leq m$.
\end{lemma}

\begin{proof}
First consider a homomorphism from one factor $B_{p,q}$.
Let $w_a,w_b,w_c,w_d\in\Z^m$ be the total directions on its four
boundary cycles.  Summing the first property over the unit squares
of the long horizontal and vertical tiles gives
\[
 q w_c=p w_d,\qquad q w_a=p w_b.
\]
Coprimality gives integer vectors $u,v$ with
\[
 w_c=pu,\quad w_d=qu,\qquad w_a=pv,\quad w_b=qv.
\]
Each edge direction has $\ell^1$-norm one, so
$\|u\|_1,\|v\|_1\leq1$.  The parity of the coordinate sum of the
total direction equals the parity of the walk length.  Since $p$
is odd, both $u,v$ have odd coordinate sum.  Hence $u,v\in S_m$.

A sum of $p$ unit coordinate vectors can equal $pu$ only if every
summand equals $u$.  Thus every positive $c,d$ boundary edge has
direction $u$, and every positive $a,b$ boundary edge has direction
$v$.  For a row of any tile, sum the first property over the strip
between that row and the top boundary.  The two vertical boundary
paths have equal total directions, because every one of their edges
has direction $v$.  A row of width $W$ therefore has total direction
$Wu$.  The same argument about equality of norms makes every edge
in that row have direction $u$.  The analogous column argument gives
direction $v$ on all positive vertical edges.  Thus the directions
are constant throughout the factor.

They cannot use the same unsigned axis.  Otherwise
$v=\varepsilon u$ for $\varepsilon\in\{1,-1\}$.  Restrict to
$C_p\square C_q$.  The two routes from $(i,j)$ to
$(i+1,j+\varepsilon)$ both have successive directions $u,u$.
The unique-middle-vertex property gives
\[
 f(i+1,j)=f(i,j+\varepsilon).
\]
Thus $f$ is invariant under $(1,-\varepsilon)$.  This translation
is transitive on $\Z/p\Z\times\Z/q\Z$, since $p,q$ are coprime.
The restriction would be constant, contrary to looplessness.

Now consider a product.  Fixing the other coordinates applies the
preceding argument to each factor slice.  The resulting two vectors
are independent of the fixed coordinates.  Indeed, for adjacent
slices $f_0,f_1:B_{p_i,q_i}\to H$, put
$\delta(b)=\omega(f_0(b),f_1(b))$.  The first property on the square
over an edge $bb'$ gives
\begin{equation}\label{eq:strict-ladder}
 \omega(f_1(b),f_1(b'))-\omega(f_0(b),f_0(b'))
 =\delta(b')-\delta(b).
\end{equation}
Sum this identity around the horizontal length-$p_i$ boundary cycle.
Its left side is $p_i$ times the difference of the two horizontal
vectors, so those vectors agree.  The vertical boundary cycle gives
the same conclusion for the vertical vectors.  Since the other
factors are connected, the vectors are independent of all their
coordinates.

It remains to exclude collapse between axes from different factors.
Restrict to the product of a boundary cycle of length $p_i$ in one
factor and one of length $p_j$ in another.  These lengths are coprime,
so the same torus argument applies.  All $2k$ vectors therefore use
distinct unsigned axes, proving the lemma.
\end{proof}

To apply these finite obstructions in higher dimensions, we need a
compactness argument.  It accounts for the passage from continuous
colorings to Cartesian products of the two-dimensional source graphs.

\begin{lemma}\label{lem:strict-compact}
For every finite loopless graph $H$ and $k\geq1$:
\begin{enumerate}
\item if $H\in\mathcal H_{2k}$, then there is a homomorphism
$B_{p_1,q_1}\square\cdots\square B_{p_k,q_k}\to H$;
\item if $H\in\mathcal H_{2k+1}$, then there is a homomorphism
$C_r\square B_{p_1,q_1}\square\cdots\square B_{p_k,q_k}\to H$
for some odd $r\geq3$.
\end{enumerate}
The parameters $p_1,q_1,\ldots,p_k,q_k$, and also $r$ in the second
statement, can all be chosen as distinct sufficiently large odd primes.
\end{lemma}

\begin{proof}
For each nonzero $g\in\Z^2$, let
\[
 P_g=\bigl\{z\in\{0,1,2\}^{\Z^2}:
 z(h)\ne z(h+g)\text{ for every }h\in\Z^2\bigr\},
 \qquad P=\prod_{g\ne0}P_g.
\]
Each $P_g$ is nonempty (color every $g$-orbit alternately with
two colors) and compact.  The product $P$ is a compact
zero-dimensional Polish space with a free $\Z^2$-action: a nonzero
$g$ cannot fix its $g$-th layer.  Every free zero-dimensional Polish
$\Z^2$-space continuously and equivariantly maps to $P$.  To see this,
choose, by Lemma~\ref{lem:clopen}, a continuous proper three-coloring
$c_g$ of each $g$-displacement graph and use the $g$-th layer
$z_g(h)=c_g(h\cdot x)$.

Continuous universality
\cite[Theorem~1.12]{Bernshteyn} transfers an $H$-coloring of the free
shift to every free zero-dimensional Polish action of the same group.
Thus $H\in\mathcal H_{2k}$ gives a continuous $H$-coloring of the
compact product $P^k$ with its product $\Z^{2k}$-action.
There are finite clopen partitions $\alpha_i$ of its factors such
that the color depends only on the tuple of partition atoms.
Indeed, cover the compact product by finitely many clopen product
rectangles on which the coloring is constant, and take the partitions
generated by their coordinate factors.

Let $J_i$ be the finite graph on the nonempty atoms of $\alpha_i$,
joining two atoms when an actual standard Schreier edge of $P$
has its endpoints in those atoms.  No atom contains such an edge:
fixing all other factor coordinates would otherwise give adjacent
product points with the same color in the loopless target $H$.
Therefore $J_i$ is loopless and $\alpha_i:P\to J_i$ is a continuous
homomorphism.  Composing with $X_2\to P$ gives
$J_i\in\mathcal H_2$.

The coloring of atom tuples defines a homomorphism
\[
 J_1\square\cdots\square J_k\longrightarrow H.
\]
For an edge in one factor, choose its witnessing edge in $P$ and
arbitrary points in the other atoms.  This verifies adjacency of
their colors.  The GJKS finite-source result now allows us to choose
$B_{p_i,q_i}\to J_i$, with all parameters distinct sufficiently
large odd primes.  Composition proves the first assertion.

For the second, use the odd odometer
\[
 O=\varprojlim_{a\geq1}\Z/3^a\Z
\]
Its points are compatible sequences of residues, and its basic
clopen sets fix a residue at one level. Translation by one
makes it a compact free $\Z$-space: an integer fixing a point
would be divisible by every power of three, and hence zero.
Universality gives an $H$-coloring of $O\times P^k$.
Compactness makes its dependence on $O$ factor through one finite
quotient $\Z/3^a\Z$, and its dependence on the other factors through
finite clopen partitions as above.  We may take $a\geq1$.
The same construction gives
\[
 C_{3^a}\square J_1\square\cdots\square J_k\longrightarrow H.
\]
Replace the $J_i$ by the corresponding twelve-tile graphs.  To make
the remaining parameter prime as well, put $s=3^a$.  For every odd
$r\geq s$, traversing $C_s$ once and inserting $(r-s)/2$ backtracking
pairs gives a closed walk of length $r$, hence a homomorphism
$C_r\to C_s$.  Choose $r$ to be an arbitrarily large odd prime
distinct from all $p_i,q_i$, and precompose with this map in the
cycle factor.  This proves the second assertion with all its
parameters prime.
\end{proof}

\begin{lemma}\label{lem:strict-odd}
Suppose $m=2k$ and $H$ satisfies the edge-direction hypotheses of
Lemma~\ref{lem:strict-product}.  For a product $B$ as in that lemma
and any odd $r\geq3$,
\[
 C_r\square B\nrightarrow H.
\]
\end{lemma}

\begin{proof}
Suppose such a homomorphism exists, with successive $C_r$-slices
$f_0,\ldots,f_r=f_0:B\to H$.  By
Lemma~\ref{lem:strict-product}, every slice has a signed permutation
matrix $M_j$.  Apply \eqref{eq:strict-ladder} to adjacent slices and
sum around each source boundary cycle.  This gives
$M_{j+1}=M_j$; call their common matrix $M$.
The same identity now makes
\[
 \delta_j(b)=\omega(f_j(b),f_{j+1}(b))
\]
independent of $b$, because $B$ is connected.  Write its value as
$\delta_j\in S_{2k}$.

For each factor choose the direction-preserving map
$\theta_i:X_2\to B_{p_i,q_i}$.  Their product is a continuous
direction-preserving homomorphism
\[
 \theta:X_2^k\longrightarrow B
\]
for the product $\Z^{2k}$-action.  Put $c_j=f_j\theta$ and
$s_j=M^{-1}\delta_j\in S_{2k}$.  The two routes
\[
 c_j(y),\ c_j(s_j\cdot y),\ c_{j+1}(s_j\cdot y)
 \quad\text{and}\quad
 c_j(y),\ c_{j+1}(y),\ c_{j+1}(s_j\cdot y)
\]
both have successive directions $\delta_j,\delta_j$.
Their middle vertices therefore agree:
\[
 c_{j+1}(y)=c_j(s_j\cdot y).
\]
Iteration gives
\[
 c_0(y)=c_0(t\cdot y),\qquad
 t=\sum_{j=0}^{r-1}s_j.
\]
The coordinate sum of $t$ is odd, so $t\ne0$.

Split $t$ into its $k$ two-dimensional components and hold fixed the
factors on which that component is zero.  On the product of the
remaining factors, translation by $t$ is topologically transitive.
In fact, sufficiently large integer multiples of each nonzero
component separate any two prescribed finite cylinder supports.
One common sufficiently large multiple works in all these factors,
and the combined finite conditions extend to free configurations.

A continuous finite-valued function invariant under a topologically
transitive homeomorphism is constant: two different nonempty color
fibers would be disjoint invariant open sets.  Consequently $c_0$
is constant on each active-factor slice with the inactive factors
fixed.  There is at least one active factor, and its standard
Schreier edges lie in that slice.  This contradicts the homomorphism
property into a loopless graph.
\end{proof}

\begin{proposition}\label{prop:strict-dimensions}
For every $d\geq1$ and $L\geq4$,
\[
 M_{d,L}\in\mathcal H_d\setminus\mathcal H_{d+1}.
\]
Hence the sequence $(\mathcal H_d)_{d\geq1}$ is strictly decreasing.
\end{proposition}

\begin{proof}
Membership in $\mathcal H_d$ follows from
Lemma~\ref{lem:strict-marker}.  If $d=2k-1$ and
$M_{d,L}\in\mathcal H_{d+1}$, the first part of
Lemma~\ref{lem:strict-compact} and
Lemma~\ref{lem:strict-product} would require $2k$ distinct target
axes in $\Z^{2k-1}$, a contradiction.
If $d=2k$ and $M_{d,L}\in\mathcal H_{d+1}$, the second part of
Lemma~\ref{lem:strict-compact} would contradict
Lemma~\ref{lem:strict-odd}.
\end{proof}

\noindent\begin{minipage}{\textwidth}
\small
\textsc{School of Mathematical Sciences and School of Pre-university,}\\
\textsc{Dalian Minzu University}\\[0.4em]
\textit{Email address:}\enspace
\href{mailto:wangruijun@dlnu.edu.cn}{\texttt{wangruijun@dlnu.edu.cn}}
\end{minipage}
\end{document}